\documentclass[12pt]{article}

\usepackage{lineno,hyperref}
\usepackage{amsmath}
\usepackage{amssymb}
\usepackage{amsthm}
\usepackage{amsfonts}
\usepackage{latexsym}
\usepackage{mathrsfs}
\usepackage{url}
\usepackage{eufrak}
\usepackage{cite}

\makeatletter
\@addtoreset{equation}{section}
\makeatother

\newtheorem{thm}{Theorem}[section]
\newtheorem{lem}[thm]{Lemma}
\newtheorem{corollary}[thm]{Corollary}
\newtheorem{proposition}[thm]{Proposition}
\theoremstyle{definition}
\newtheorem{definition}[thm]{Definition}
\newtheorem{notation}[thm]{Notation}
\newtheorem{rmk}[thm]{Remark}

\def\CC{\mathcal{C}}
\def\I{\mathcal{I}}
\def\R{\mathcal{R}}
\def\T{\mathcal{T}}
\def\C{\mathscr{C}}
\def\M{\mathscr{M}}
\def\N{\mathscr{N}}
\def\PG{\mathrm{PG}}
\def\F{\mathbb{F}}
\def\P{\mathbf{P}}

\begin{document}
\title{On planes through points off the twisted cubic in $\PG(3,q)$ and multiple covering codes
\author{Daniele Bartoli\footnote{The research of  D. Bartoli, S. Marcugini, and F.~Pambianco was
 supported in part by the Italian
National Group for Algebraic and Geometric Structures and their Applications (GNSAGA - INDAM) and by
University of Perugia (Project: Curve algebriche in caratteristica positiva e applicazioni, Base Research
Fund 2018).} \\
{\footnotesize Dipartimento di Matematica e Informatica,
Universit\`{a}
degli Studi di Perugia, }\\
{\footnotesize Via Vanvitelli~1, Perugia, 06123, Italy. E-mail:
daniele.bartoli@unipg.it}
\and Alexander A. Davydov\footnote{The research of A.A.~Davydov was done at IITP RAS and supported by the Russian Government (Contract No 14.W03.31.0019).} \\
{\footnotesize Institute for Information Transmission Problems
(Kharkevich
institute), Russian Academy of Sciences}\\
{\footnotesize Bol'shoi Karetnyi per. 19, Moscow,
127051, Russian Federation. E-mail: adav@iitp.ru}
\and Stefano Marcugini$^*$ and Fernanda Pambianco$^*$  \\
{\footnotesize Dipartimento di Matematica e Informatica,
Universit\`{a}
degli Studi di Perugia, }\\
{\footnotesize Via Vanvitelli~1, Perugia, 06123, Italy. E-mail:
\{stefano.marcugini,fernanda.pambianco\}@unipg.it}}
}
\date{}
\maketitle

\textbf{Abstract.}
Let $\PG(3,q)$ be the projective space of dimension three over the finite field with $q$ elements.  Consider a twisted cubic in $\PG(3,q)$. The structure of the point-plane incidence matrix  in $\PG(3,q)$ with respect to the orbits of points and planes under the action of the stabilizer group of the twisted cubic is described. This information is used to view
generalized doubly-extended Reed-Solomon codes of codimension four as asymptotically optimal multiple covering codes.



\section{Introduction}
Let $\F_{q}$ be the Galois field with $q$ elements, $\F_{q}^*=\F_{q}\setminus\{0\}$, $\F_q^+=\F_q\cup\{\infty\}$.
Let $\PG(N,q)$ be the $N$-dimensional projective space over $\F_q$; it contains $\theta_{N,q}=(q^{N+1}-1)/(q-1)$   points. We denote by  $[n,k,d]_{q}R$ an ${\mathbb{F}}_{q}$-linear code of length $n$, dimension $k$,
minimum distance~$d$, and covering radius $R$.
 For an introduction to projective spaces over finite fields and connections between  projective geometry and
coding theory see \cite{Hirs_PGFF,Hirs_PG3q,HirsStor-2001,HirsThas-2015,EtzStorm2016}.

An $n$-arc in  $\PG(N,q)$, with $n\ge N + 1\ge3$, is a
set of $n$ points such that no $N +1$ points belong to
the same hyperplane of $\PG(N,q)$. An $n$-arc is complete if it is not contained in an $(n+1)$-arc. Arcs and linear maximum distance separable (MDS) $[n,k,n-k+1]_{q}$  codes are equivalent objects, see e.g. \cite{EtzStorm2016,MWS}.

In $\PG(N,q)$, $2\le N\le q-2$, a normal rational curve is any $(q+1)$-arc projectively equivalent to the arc
$\{(t^N,t^{N-1},\ldots,t^2,t,1):t\in \F_q\}\cup \{(1,0,\ldots,0)\}$.  The points (in homogeneous coordinates) of a normal rational curve in $\PG(N,q)$
treated as columns define a parity check matrix of a $[q + 1,q-N,N + 2]_q$ generalized doubly-extended Reed-Solomon (GDRS) code \cite{EtzStorm2016,Roth}. Clearly, a GDRS code is MDS. In $\PG(3,q)$, the normal rational curve is called a  \emph{twisted cubic} \cite{Hirs_PG3q,HirsThas-2015}. Twisted cubics have important connections with a number of other objects, see e.g. \cite{BonPolvTwCub,BrHirsTwCub,CLPolvT_Spr,CosHirsStTwCub,GiulVincTwCub,Hirs_PG3q,HirsStor-2001,HirsThas-2015,LunarPolv} and the references therein.

Twisted cubics in $\PG(3,q)$ have been widely studied; see  \cite{Hirs_PG3q} and the references therein.  In particular, in \cite{Hirs_PG3q}, the orbits of planes and points under the group of the projectivities fixing a cubic are considered.

In this paper we investigate the intersection multiplicities of planes and twisted cubics, determining  the structure of the point-plane incidence matrix  in $\PG(3,q)$. As a by-product, we also give a number of useful relations regarding these numbers.

As an application,  we show that  twisted cubics can be treated as  multiple $\rho$-saturating sets with $\rho=2$ which, in turn, give rise to  asymptotically optimal  non-binary linear multiple covering $[q+1,q-3,5]_q3$ codes of radius $R=3$. Thereby, we show that the $[q+1,q-3,5]_q3$ GDRS code associated with the twisted cubic can be viewed as an asymptotically optimal  multiple covering. Note that
 in the literature, see e.g. \cite{BDGMP_MultCov,BDGMP_MultCov2016,CHLL_CovCodBook,PDBGM-ExtendAbstr}, several examples of multiple coverings with $R=2$ and $\rho=1$ are given whereas asymptotically optimal multiple coverings with $R=3$ and $\rho=2$ are not considered.

The paper is organized as follows. Section \ref{sec_prelimin} contains preliminaries. In Section \ref{sec_main_res}, the main results of the paper are presented. Section \ref{sec_useful} provides  a number useful relations. In Sections \ref{sec_exac_res} and \ref{sec_points in planes} we compute the spectrum of the intersections between planes and twisted cubics,  and the structure of the point-plane incidence matrix  in $\PG(3,q)$ is described. Covering properties of the codes associated with twisted cubics are considered in Section~\ref{sec_multipl}.

\section{Preliminaries}\label{sec_prelimin}

For the convenience of readers, in this section we summarize known results on twisted cubics \cite[Chapter 21]{Hirs_PG3q} and on multiple covering codes \cite{BDGMP_MultCov,BDGMP_MultCov2016,CHLL_CovCodBook,PDBGM-ExtendAbstr}.
\subsection{Twisted cubic}\label{subset_twis_cub}
Let $\P(x_0,x_1,x_2,x_3)$ be a point of $\PG(3,q)$ with the homogeneous coordinates $x_i\in\F_{q}$; the rightmost nonzero coordinate is equal to $1$. For $t\in\F_q^+$, let  $P(t)$ be a point such that
\begin{align*}
&P(t)=\P(t^3,t^2,t,1)\text{ if }t\in\F_q,\displaybreak[3] \\
&P(\infty)=\P(1,0,0,0).
\end{align*}
Let $\C\subset\PG(3,q)$ be the \emph{twisted cubic} consisting of $q+1$ points $P_1,\ldots,P_{q+1}$ no four of which are coplanar.
We consider $\C$ in the canonical form
\begin{align}\label{eq2_cubic}
\C=\{P_1,P_2,\ldots,P_{q+1}\}=\{P(t)|t\in\F_q^+\}.
\end{align}
Let $\boldsymbol{\pi}(c_0,c_1,c_2,c_3)\subset\PG(3,q)$, $c_i\in\F_q$, be the plane with equation
$c_0x_0+c_1x_1+c_2x_2+c_3x_3=0.$
The plane through three points $P(t_1),P(t_2),P(t_3)$ of $\C$ is
\begin{align}\label{eq2_plane3points}
    \boldsymbol{\pi}(1,-(t_1+t_2+t_3),t_1t_2+t_1t_3+t_2t_3,-t_1t_2t_3)\supset\{P(t_1),P(t_2),P(t_3)\}.
\end{align}
When the three points coincide with each other and $t_1=t_2=t_3=t$, we have, through the
 point $P(t)=\P(t^3,t^2,t,1)\in\C$, an \emph{osculating plane} $\pi_\text{osc}(t)$ such that
\begin{align}\label{eq2_osc pl}
&\pi_\text{osc}(t)=\boldsymbol{\pi}(1,-3t,3t^2,-t^3), ~~P(t)=\P(t^3,t^2,t,1)\in\pi_\text{osc}(t);\displaybreak[3]\\
&\pi_\text{osc}(\infty)=\boldsymbol{\pi}(0,0,0,1), ~~P(\infty)=\P(1,0,0,0)\in\pi_\text{osc}(\infty).\label{eq2_osc pl_inf}
\end{align}
The osculating plane $\pi_\text{osc}(t)$ meets $\C$ only in $P(t)$.
The osculating planes form the \emph{osculating developable} to $\C$, that is, a pencil of planes for $q\equiv0\pmod3$ or a cubic developable otherwise.

A \emph{chord} of $\C$ is a line through a pair of real points of $\C$ or a pair of complex conjugate points. In the last  case it is an \emph{imaginary chord}. If the real points are distinct, it is a \emph{real chord}. If the real points coincide with each other, it is a \emph{tangent.}

\begin{notation}\label{notation_1}
The following notation is used:
\begin{align*}
  &G_q && \text{the group of projectivities in } \PG(3,q) \text{ fixing }\C;\displaybreak[3] \\
  &\mathbf{Z}_n&&\text{cyclic group of order }n;\displaybreak[3] \\
  &\mathbf{S}_n&&\text{symmetric group of degree }n;\displaybreak[3] \\
  &\Gamma &&\text{the osculating developable to }  \C;\displaybreak[3]\\
  &\mathfrak{A}&&\text{the null polarity \cite[Chapter 2.1.5]{Hirs_PGFF}, \cite[Theorem 21.1.2]{Hirs_PG3q};}\displaybreak[3]\\
&\Gamma\text{-plane}  &&\text{an osculating plane of }\Gamma;\displaybreak[3]\\
&d_\C\text{-plane}&&\text{a plane containing \emph{exactly} $d$ distinct points of }\C,~d=0,1,2,3;\displaybreak[3]\\
&1_{\C}\setminus\Gamma\text{-plane}&&\text{a $1_\C$-plane not in $\Gamma$;}\displaybreak[3]\\
&\C\text{-point}&&\text{a point  of $\C$;}\displaybreak[3]\\
&\mu_\Gamma\text{-point}&&\text{a point  off $\C$ lying on \emph{exactly} $\mu$ osculating planes, }\mu_\Gamma=0,1,3,q+1;\displaybreak[3]\\
&\text{T-point}&&\text{a point  off $\C$  on a tangent to $\C$ for $q\not\equiv0\pmod 3$;}\displaybreak[3]\\
&\text{TO-point}&&\text{a point  off $\C$ on a tangent and one osculating plane for }q\equiv 0\pmod 3;\displaybreak[3]\\
&\text{RC-point}&&\text{a point  off $\C$  on a real chord;}\displaybreak[3]\\
&\text{IC-point}&&\text{a point  on an imaginary chord;}\displaybreak[3]\\
&A^{tr}&&\text{the transposed matrix }A;\displaybreak[3]\\
&\#S&&\text{the cardinality of a set }S;\displaybreak[3]\\
&t\text{-}(v,k,\lambda)&&\text{design on a set $V$ of $v$ ``points'' in which a ``block'' is a $k$-subset of $V$}\displaybreak[3]\\
&&&\text{and every $t$-subset of $V$ is contained  in exactly $\lambda$ blocks.}
\end{align*}
\end{notation}

The following theorem summarizes known results from \cite{Hirs_PG3q}.
\begin{thm}\label{th2_HirsCor4,5}
\emph{\cite[Chapter 21]{Hirs_PG3q}} The following properties of the twisted cubic $\C$ of \eqref{eq2_cubic} hold:

\textbf{\emph{A.}}
The group $G_q$ acts triply transitively on $\C$. Also,
\begin{align*}
& G_q\cong PGL(2,q),&& \text{ for }q\ge5; \displaybreak[3]\\
 & G_4\cong\mathbf{S}_5\cong P\Gamma L(2,4),&&\#G_4=2\cdot\#PGL(2,4)=120;\displaybreak[3]\\
 & G_3\cong\mathbf{S}_4\mathbf{Z}_2^3,&&\#G_3=8\cdot\#PGL(2,3)=192;\displaybreak[3]\\
 & G_2\cong\mathbf{S}_3\mathbf{Z}_2^3,&&\#G_2=8\cdot\#PGL(2,2)=48.
\end{align*}

\textbf{\emph{B.}} Let $q\ge5$. Under $G_q$, there are five orbits $\N_i$ of planes and five orbits $\M_j$ of points. These orbits have the following properties:

\textbf{\emph{(i)}} For all $q$, the orbits $\N_i$ of planes are as follows:
 \begin{align}\label{eq2_plane orbit_gen}
   &\N_1=\{\Gamma\text{-planes}\},~\#\N_1=q+1;~\N_{2}=\{2_\C\text{-planes}\}, ~\#\N_{2}=q(q+1);\displaybreak[3]\\
 &\N_{3}=\{3_\C\text{-planes}\}, ~  \#\N_{3}=\frac{q(q^2-1)}{6};~\N_{4}=\{1_\C\setminus\Gamma\text{-planes}\},~\#\N_{4}= \frac{q(q^2-1)}{2};\displaybreak[3]\notag\\
 &\N_{5}=\{0_\C\text{-planes}\},~\#\N_{5}=\frac{q(q^2-1)}{3} \,.\notag
 \end{align}

\textbf{\emph{(ii)}} For $q\not\equiv0\pmod 3$, the orbits $\M_j$ of points are as follows:
\begin{align}\label{eq2_point_orbits_gen}
&\M_1=\C,~\#\M_1=q+1;~ \M_2=\{\text{\emph{T}}\text{-points}\},~\#\M_2=q(q+1);\displaybreak[3]\\
& \M_3=\{3_\Gamma\text{-points}\}, ~\#\M_{3}=\frac{q(q^2-1)}{6};~\M_4=\{1_\Gamma\text{-points}\},~\#\M_{4}=\frac{q(q^2-1)}{2};\displaybreak[3]\notag\\
&\M_5=\{0_\Gamma\text{-points}\}, ~\#\M_{5}=\frac{q(q^2-1)}{3}\,.\notag
\end{align}
Also,
\begin{align}\label{eq2_=1_orbit_point}
 &\text{if } q\equiv1\pmod 3 \text{ then } \M_{3}\cup\M_{5}=\{\text{\emph{RC}-points}\}, ~\M_{4}=\{\text{\emph{IC}-points}\};\displaybreak[3]\\
 &\text{if } q\equiv-1\pmod 3\text{ then }\M_{3}\cup\M_{5}=\{\text{\emph{IC}-points}\},~
 \M_{4}=\{\text{\emph{RC}-points}\}.\label{eq2_=2_orbit_point}
\end{align}

\textbf{\emph{(iii)}} For $q\equiv0\pmod 3$, the orbits $\M_k$ of points are as follows:
\begin{align}\label{eq2_point_orbits_j=0}
&\M_1=\C,~\#\M_1=q+1;~\M_2=\{(q+1)_\Gamma\text{-points}\},~\#\M_2=q+1;\displaybreak[3]\\
&\M_3=\{\text{\emph{TO}-points}\},~\#\M_3=q^2-1;~\M_4=\{\text{\emph{RC}-points}\},~\#\M_4=\frac{q(q^2-1)}{2};\notag\displaybreak[3]\\
& \M_5=\{\text{\emph{IC}-points}\}, ~\#\M_5=\frac{q(q^2-1)}{2}\,. \notag
\end{align}

\textbf{\emph{C.}} \textbf{\emph{(i)}} In total, there are $\binom{q+1}{2}$ real chords of $\C$,  $q+1$ tangents to $\C$, and $\binom{q}{2}$ imaginary chords of $\C$.

\textbf{\emph{(ii)}} No two chords of $\C$ meet off $\C$. Every point off $\C$ lies on exactly one chord of~$\C$.

\textbf{\emph{D.}}  For $q\not\equiv0\pmod 3$, the null polarity $\mathfrak{A}$ interchanges $\C$ and $\Gamma$; also,
\begin{align}\label{eq2_null_pol}
    \M_i\mathfrak{A}=\N_i, ~\#\M_i=\#\N_i,~i=1,\ldots,5.
\end{align}
\end{thm}

\begin{rmk}\label{rem2_q+1 osc}
  For  $q\equiv0\pmod3$, $\Gamma$ is a pencil of $q+1$ planes, see \cite[Theorem 21.1.2(i)]{Hirs_PG3q}. Points lying on all these planes (the orbit $\M_2$) form a line external to $\C$. All $d_\C$-planes with $d=0,1,2,3$ intersect this line.
\end{rmk}
\subsection{The point-plane incidence matrix  of $\PG(3,q)$}\label{subsec_incid}

Let  $\I$ be the $\theta_{3,q}\times\theta_{3,q}$ point-plane incidence matrix  of $\PG(3,q)$ in which columns correspond to points, rows correspond to planes, and there an entry is  ``1'' iff the corresponding point belongs to the corresponding plane. Every column and every row of $\I$ contains exactly $\theta_{2,q}$ ones, i.e. $\I$ is a tactical configuration \cite[Chapter 2.3]{Hirs_PGFF}.
Moreover, $\I$ gives a symmetric 2-$(\theta_{3,q},\theta_{2,q},q+1)$ design as there are exactly $q+1$ planes through any two points of $\PG(3,q)$.

For  $q\ge5$,  orbits $\N_i$ and $\M_j$ partition $\I$ into 25 submatrices $\I_{ij}$, with $i,j=1,\ldots,5$, where $\I_{ij}$ has size $\#\N_i\times\#\M_j$.

It is clear (see Lemma \ref{lem4_pointinplane}) that every plane of $\N_i$ contains the same number of points from  $\M_j$;  we denote this number as $k_{ij}$. And vice versa, through every point of $\M_j$ we have the same number of planes from~$\N_i$; we denote this number as $r_{ij}$. This means that $\I_{ij}$ contains $k_{ij}$ ones in each row and $r_{ij}$ ones in each column, i.e. $\I_{ij}$ \emph{is a tactical configuration}.

Tactical configurations are useful in several distinct areas, in particular, to construct bipartite graph codes, see e.g.  \cite{BargZem,DGMP_BipGraph,HohJust} and the references therein.

\subsection{Linear multiple covering codes and multiple saturating sets}
 Let $\F_{q}^{n}$ be
the space of $n$-dimensional vectors over $\F_{q}$. Consider a linear code $C\subseteq \F_{q}^{n}$ and denote by $A_{w}(C)$ the number of its codewords of weight $w$. Let $d(x,c)$ be the Hamming distance between vectors $x$ and $c$ of~${\mathbb{F}}_{q}^{n}$ and denote by $d(x,C)=\min\limits_{c\in C}d(x,c)$ the distance between  $x$ and $C$.

\begin{definition}\label{def2_MCFcode}
\cite{BDGMP_MultCov,CHLL_CovCodBook,PDBGM-ExtendAbstr} An $[n,k,d]_{q}R$ code $C$ is  an
\emph{$(R,\mu )$ multiple covering of the farthest-off points} $((R,\mu)$-MCF code for short) if for all $x\in {\mathbb{F}}_{q}^{n}$ such that $
d(x,C)=R$ the number of codewords $c$ such that $d(x,c)=R$ is at least $\mu $.
\end{definition}

In the literature, MCF codes are also called \emph{multiple coverings of deep holes}.

The covering quality of an $[n,k,d(C)]_{q}R$ MCF code $C$ is characterized by its \emph{$\mu$-density} $\gamma _{\mu }(C,R,q)\ge1$ so that
\begin{align}
&\gamma _{\mu }(C,R,q)=\frac{{\binom{n}{{R}}}(q-1)^{R}-{\binom{2R-1}{{R-1}
}} A_{2R-1}(C)}{\mu  \left(q^{n-k}-\sum\limits_{i=0}^{R-1}{\binom{n}{{i}}}(q-1)^{i}\right)}~\text{ if }~d(C)\geq
2R-1,  \label{eq2_mudens_d>=2R-1}
\end{align}
see  \cite[Proposition 2.3]{BDGMP_MultCov}, \cite[Proposition 1]{BDGMP_MultCov2016}.
From the covering problem point of view, the best codes are those with small $\mu$-density.
If $\gamma _{\mu }(C,R,q)=1$ then  $C$ is called perfect MCF code. If the $\mu$-density $\gamma_\mu(C,R,q)$ tends to 1 when $q$ tends to infinity we have an \emph{asymptotically optimal collection of MCF codes} or, in another words, an \emph{asymptotically optimal multiple covering}.

\begin{definition}\label{def2_M1-M3}
\cite{BDGMP_MultCov,PDBGM-ExtendAbstr}
Let $S$ be an $n$-subset of points of $\PG(N,q)$. Then $S$
is said to be $(\rho ,\mu )$-saturating if:
\begin{description}
\item[(M1)] $S$ generates $\PG(N,q)$;
\item[(M2)] there exists a point $Q$ in $\PG(N,q)$ which does not belong to
any subspace of dimension $\rho -1$ generated by the points of $S$;
\item[(M3)] every point $Q$ in $\PG(N,q)$ not belonging to any subspace of
dimension $\rho -1$ generated by the points of $S$ is such that the number
of subspaces of dimension $\rho $ generated by the points of $S$ and
containing $Q$ is at least $\mu $.
\end{description}
\end{definition}
Here we slightly simplified the corresponding definition of \cite{BDGMP_MultCov,PDBGM-ExtendAbstr}.
\begin{definition}
\label{def_minimal saturating set}\cite{BDGMP_MultCov} A $(\rho ,\mu )$-saturating $n$-set in $\PG(N,q)$ is called \emph{minimal }if it does not contain a $(\rho ,\mu )$-saturating $(n-1)$-set in $\PG(N,q)$.
\end{definition}

\begin{proposition}\label{prop2_connection}
   \emph{\cite[Proposition 3.6]{BDGMP_MultCov}} Let $S$  be a $(\rho ,\mu )$-saturating  $n$-set in $\PG(n-k-1, q)$. Let
a linear $[n,k]_qR$ code $C$  admit a parity-check matrix whose
columns are homogeneous coordinates of the points in $S$. Then $C$  is a $(\rho+1 ,\mu )$-MCF code.
\end{proposition}

Proposition \ref{prop2_connection} allows us to consider $(\rho ,\mu )$-saturating sets as linear $(\rho+1 ,\mu )$-MCF codes and vice versa.

\begin{rmk}
 An $[n,k,d]_{q}R$ MCF code provides multiple coverings of the \emph{farthest-off points} or \emph{deep holes}, i.e. vectors of $\F_{q}^{n}$ lying on distance $R$ from the code. There are the useful relations of the deep holes with bounds on the size of the lists in the list decoding~\cite[Chapter 9]{Roth} of generalized Reed-Solomon codes, see e.g.\ \cite{HongWu2016,JusHoh2001,KaipaIEEE2017,XuXu2019,ZDK_DHRSarXiv2019} and the references therein. In particular, in \cite{KaipaIEEE2017,ZDK_DHRSarXiv2019}, the classification of deep holes of Reed-Solomon codes with redundancy 3 and 4 is considered.
\end{rmk}

\section{Main results}\label{sec_main_res}
From now on we consider $q\ge5$ apart from Theorems \ref{th3_main}(B) and \ref{th6_q_2_3_4}.

Tables \ref{tab1} and \ref{tab2} and Theorem \ref{th3_main} summarize the results of  Sections \ref{sec_useful}--\ref{sec_points in planes}.

In particular, for the point-plane incidence matrix, Tables \ref{tab1} and \ref{tab2} show values $k_{ij}$ (top entry) and $r_{ij}$ (bottom entry) for each possible pair $(\N_i,\M_j)$, where $k_{ij}$  is the number of points from $\M_j$ in  every plane of $\N_i$, whereas $r_{ij}$ is the number of planes from $\N_i$ through every point of $\M_j$. In other words, $k_{ij}$  (resp. $r_{ij}$) is the number of ones in every row (resp. column) of the $\#\N_i\times\#\M_j$ submatrix $\I_{ij}$ of the point-plane incidence matrix.

\begin{table}[h]
\caption{Values $k_{ij}$ (the number of ones in every row, top entry) and $r_{ij}$ (the number of ones in every column, bottom entry) for the $\#\N_i\times\#\M_j$ submatrices $\I_{ij}$ of the point-plane incidence matrix  of $\PG(3,q)$, $q\equiv \xi\pmod3$, $\xi=-1,1$, $q\ge5$\smallskip}
$
\begin{array}{@{}c|c|c||c|c|c|c|c@{}}
  \hline
&\multicolumn{2}{c||}{\M_j\rightarrow}&\M_1&\M_2&\M_3&\M_4&\M_5  \\
 \N_i&\multicolumn{2}{c||}{} &\C\text{-points}&\text{T-points}
&3_\Gamma\text{-points}&1_\Gamma\text{-points}&0_\Gamma\text{-points}\\
\downarrow&\multicolumn{2}{c||}{}&q+1&q^2+q&\frac{1}{6}(q^3-q)&\frac{1}{2}(q^3-q)&\frac{1}{3}(q^3-q)\vphantom{H_{H_{H_{H_H}}}}\\\hline\hline
\N_1&\Gamma\text{-planes}&k_{1j}&1 &2q&\frac{1}{2}(q^2-q)&\frac{1}{2}(q^2-q)&0\vphantom{H^{H^{H^a}}}\\
&q+1&r_{1j}&1&2&3&1&0\\\hline
\N_2& 2_\C\text{-planes}& k_{2j}&2&2q-1&\frac{1}{6}(q^2-3q+2)&\frac{1}{2}(q^2-q)&\frac{1}{3}(q^2-1)\vphantom{H^{H^{H^a}}}\\
&q^2+q&r_{2j}&2q&2q-1&q-2&q&q+1\\\hline
\N_3&3_\C\text{-planes}&k_{3j}&3&q-2&\frac{1}{6}(q^2+\xi q+4)&\frac{1}{2}(q^2-\xi q)&\frac{1}{3}(q^2+\xi q-2)\vphantom{H^{H^{H^a}}}\\
&\frac{1}{6}(q^3-q)&r_{3j}&\frac{1}{2}(q^2-q)&\frac{1}{6}(q^2-3q+2)& \frac{1}{6}(q^2+\xi q+4)&\frac{1}{6}(q^2-\xi q)&\frac{1}{6}(q^2+\xi q-2)\vphantom{H^{H^{H^H}}}\vphantom{H_{H_{H_{H_H}}}}\\\hline
\N_4&1_\C\setminus\Gamma\text{-}&k_{4j}&1&q&\frac{1}{6}(q^2-\xi q)&\frac{1}{2}(q^2+\xi q)&\frac{1}{3}(q^2-\xi q)\vphantom{H^{H^{H^a}}}\\
&\text{planes}&&&&&\\
&\frac{1}{2}(q^3-q)&r_{4j}&\frac{1}{2}(q^2-q)&\frac{1}{2}(q^2-q)&\frac{1}{2}(q^2-\xi q)&\frac{1}{2}(q^2+\xi q)&\frac{1}{2}(q^2-\xi q)\vphantom{H_{H_{H_{H_H}}}}\\\hline
\N_5&0_\C\text{-planes} &k_{5j}&0&q+1&\frac{1}{6}(q^2+\xi q-2)&\frac{1}{2}(q^2-\xi q)&\frac{1}{3}(q^2+\xi q+1)\vphantom{H_{H_{H_H}}}\vphantom{H^{H^{H^a}}}\\
&\frac{1}{3}(q^3-q)&r_{5j}&0&\frac{1}{3}(q^2-1)&\frac{1}{3}(q^2+\xi q-2)&\frac{1}{3}(q^2-\xi q)&\frac{1}{3}(q^2+\xi q+1)\vphantom{H^{H^{H^H}}}\vphantom{H_{H_{H_{H_H}}}}\\\hline
\end{array}
$
\label{tab1}
\end{table}

\begin{table}[h]
\caption{Values $k_{ij}$ (the number of ones in every row, top entry) and $r_{ij}$ (the number of ones in every column, bottom entry) for the $\#\N_i\times\#\M_j$ submatrices $\I_{ij}$ of the point-plane incidence matrix  of $\PG(3,q)$, $q\equiv 0\pmod3$, $q\ge5$\smallskip}
$
\begin{array}{c|c|c||c|c|c|c|c}
  \hline
&\multicolumn{2}{c||}{\M_j\rightarrow}&\M_1&\M_2&\M_3&\M_4&\M_5  \\
&\multicolumn{2}{c||}{}&\C\text{-points}&(q+1)_\Gamma
&\text{TO-points}&\text{RC-points}&\text{IC-points}\\
\N_i&\multicolumn{2}{c||}{}&&\text{-points}&&\\
\downarrow&\multicolumn{2}{c||}{}&q+1&q+1&q^2-1&\frac{1}{2}(q^3-q)&\frac{1}{2}(q^3-q)\vphantom{H_{H_{H_{H_H}}}}\\\hline\hline
\N_1&\Gamma\text{-planes}&k_{1j}&1 &q+1&q-1&\frac{1}{2}(q^2-q)&\frac{1}{2}(q^2-q)\vphantom{H^{H^{H^a}}}\\
&q+1&r_{1j}&1&q+1&1&1&1\\\hline
\N_2& 2_\C\text{-planes}&k_{2j}&2&1&2q-2&\frac{1}{2}(q^2-q)&\frac{1}{2}(q^2-q)\vphantom{H^{H^{H^a}}}\\
&q^2+q&r_{2j}&2q&q&2q&q&q\\\hline
\N_3&3_\C\text{-planes}&k_{3j}&3&1&q-3&\frac{1}{2}(q^2+q)&\frac{1}{2}(q^2-q)\vphantom{H^{H^{H^a}}}\\
&\frac{1}{6}(q^3-q)&r_{3j}&\frac{1}{2}(q^2-q)&\frac{1}{6}(q^2-q)&\frac{1}{6}(q^2-3q)&
\frac{1}{6}(q^2+q)&\frac{1}{6}(q^2-q)\vphantom{H^{H^{H^H}}}\vphantom{H_{H_{H_{H_H}}}}\\\hline
\N_4&1_\C\setminus\Gamma\text{-planes}& k_{4j}&1&1&q-1&\frac{1}{2}(q^2-q)&\frac{1}{2}(q^2+q)\vphantom{H^{H^{H^a}}}\\
&\frac{1}{2}(q^3-q)&r_{4j}&\frac{1}{2}(q^2-q)&\frac{1}{2}(q^2-q)&\frac{1}{2}(q^2-q)&
\frac{1}{2}(q^2-q)&\frac{1}{2}(q^2+q)\vphantom{H^{H^{H^H}}}\vphantom{H_{H_{H_{H_H}}}}\\\hline
\N_5&0_\C\text{-planes} &k_{5j}&0&1&q&\frac{1}{2}(q^2+q)&\frac{1}{2}(q^2-q)\vphantom{H^{H^{H^a}}}\\
&\frac{1}{3}(q^3-q)&r_{5j}&0&\frac{1}{3}(q^2-q)&\frac{1}{3}q^2&
\frac{1}{3}(q^2+q)&\frac{1}{3}(q^2-q)\vphantom{H^{H^{H^H}}}\vphantom{H_{H_{H_{H_H}}}}\\\hline
\end{array}
$
\label{tab2}
\end{table}

\begin{thm}\label{th3_main} \textbf{\emph{A.}} Let $q\ge5$. Let $q\equiv \xi\pmod3$. The following holds:
\begin{description}
  \item[(i)]
In $\PG(3,q)$, let notations of planes, points, and incidence submatrices be as in Sections \emph{\ref{subset_twis_cub}} and \emph{\ref{subsec_incid}}. Then, for the point-plane incidence matrix, the values $k_{ij}$ (i.e. the number of distinct points in distinct planes) and $r_{ij}$ (i.e. the number of distinct planes through distinct points) are given by Tables \emph{\ref{tab1}} and \emph{\ref{tab2}}.
  \item[(ii)]
  Up to rearrangement of rows and columns, we have
  \begin{align*}
  &  \I_{ij}^{\,tr}=\I_{ji},~k_{ij}=r_{ji}, ~r_{ij}=k_{ji},~i,j=1,\ldots,5,~\text{for }\xi\ne0;\displaybreak[3]\\
  & \I_{41}^{tr}=\I_{14},~ \I_{41}^{tr}=\I_{15},~ \I_{42}^{tr}=\I_{14},~ \I_{42}^{tr}=\I_{15},~\text{for }\xi=0;\displaybreak[3]\\
 &\I_{i4}\text{ for $\xi=1$ is the same as $\I_{i5}$ for }\xi=0,~ i=1,\ldots,5;\displaybreak[3]\\
  &\I_{i4}\text{ for $\xi=-1$ and for $\xi=0$ is the same, $i=1,\ldots,5$}.
  \end{align*}
  \item[(iii)] The submatrix $\I_{21}$ gives a $2\text{-}(q+1,2,2)$ design and the submatrix $\I_{31}$ defines $3\text{-}(q+1,3,1)$ and $2\text{-}(q+1,3,q-1)$ designs.
\end{description}

\textbf{\emph{B.}} Let $q=2,3,4$. Then the point-plane incidence matrix can be represented as in Tables \emph{\ref{tab1}} and \emph{\ref{tab2}} if $\N_i,\M_j$ are orbits under a group isomorphic to $\mathbf{S}_{q+1}$, where $\mathbf{S}_{q+1}$ is isomorphic to a subgroup of $G_q$ for $q=2,3$, whereas $\mathbf{S}_{4+1}\cong G_4$, cf. Theorem \emph{\ref{th6_q_2_3_4}}.
\end{thm}

Theorem \ref{th3_mu} summarizes the results of Section \ref{sec_multipl}.

\begin{thm}\label{th3_mu}
Let
\begin{align}\label{eq3_mu}
    \mu=\left\{
\begin{array}{ccc}
\frac{q^2-3q+2}{6} & \text{if} & q\not\equiv 0\pmod 3\smallskip \\
\frac{q^2-3q}{6} & \text{if} & q\equiv 0\pmod 3
    \end{array}
    \right..
\end{align}
\begin{description}
  \item[(i)]
The twisted cubic $\C$ of \eqref{eq2_cubic} is a minimal $(2,\mu )$-saturating $(q+1)$-set.
  \item[(ii)] The generalized doubly-extended Reed-Solomon code $C$ associated with the twisted cubic $\C$ of \eqref{eq2_cubic} is a $(3,\mu )$ multiple covering of the farthest-off points, i.e. $(3,\mu)$-MCF code,
     with parameters $[q+1,q-3,5]_q3$.
      Its $\mu$-density $\gamma _\mu(C,3,q)$ tends  to~1 from above when $q$ tends to infinity; thereby, we have an asymptotical optimal collection of MCF codes.
\end{description}
\end{thm}

\section{Some useful relations}\label{sec_useful}
\begin{notation}\label{not1b} Let $d\in\{0,1,2,3\}$.
The following notation is used:
\begin{align*}
   &n_{d}^\Sigma&&\text{the total number of $d_\C$-planes;} \\
   &n_{d,\C}&&\text{the number of $d_\C$-planes through a $\C$-point;} \\
   &n_{d}(A)&&\text{the number of $d_\C$-planes through a point $A$;}\displaybreak[3] \\
   &n_{d,\mu_\Gamma}^{(\xi)}&&\text{the number of $d_\C$-planes through a $\mu_\Gamma$-point for $q\equiv \xi\pmod 3$}\displaybreak[3] \\
   &&& \text{where }\mu_\Gamma\in\{0,1,3\}\text{ if $\xi\ne0$ and $\mu_\Gamma=q+1$ if $\xi=0$;}\displaybreak[3] \\
   &n_{d,\text{T}}^{(\ne0)}&&\text{the number of $d_\C$-planes through a T-point for $q\not\equiv0\pmod 3$;}\displaybreak[3] \\
   &n_{d,\text{TO}}^{(0)}&&\text{the number of $d_\C$-planes through a TO-point for $q\equiv0\pmod 3$;}\displaybreak[3] \\
   &n_{d,\text{RC}}^{(0)}&&\text{the number of $d_\C$-planes through an RC-point for $q\equiv0\pmod 3$;}\displaybreak[3]\\
   &n_{d,\text{IC}}^{(0)}&&\text{the number of $d_\C$-planes through an IC-point for $q\equiv0\pmod 3$.}
\end{align*}
\end{notation}

\begin{rmk}
In Notation \ref{not1b}, the values $n_{d,\bullet}^{(\star)}$ are equal to the parameters $r_{ij}$ of the submatrices $\I_{ij}$.  Using numbers of orbits in Theorem \ref{th2_HirsCor4,5}(B) and Tables \ref{tab1} and \ref{tab2}, one can easy set the correspondence between $n_{d,\bullet}^{(\star)}$ and $r_{ij}$. For example,
\begin{align*}
&n_{0,\C}=r_{5,1},~n_{1,\C}=r_{1,1}+r_{4,1},~n_{2,\C}=r_{2,1},~n_{3,\C}=r_{3,1};\displaybreak[3]\\
&n_{0,0_\Gamma}^{(\xi)}=r_{5,5},~n_{1,0_\Gamma}^{(\xi)}=r_{1,5}+r_{4,5},
~n_{2,0_\Gamma}^{(\xi)}=r_{2,5},~n_{3,0_\Gamma}^{(\xi)}=r_{3,5},~q\equiv\xi\pmod 3,~\xi\ne0.
\end{align*}
\end{rmk}

\begin{lem}\label{lem4_3 2}
 For all $q$, the number of\/ $3_\C$-planes and $2_\C$-planes through a real chord of $\C$ is equal to $q-1$ and $2$, respectively.
\end{lem}
\begin{proof} We consider the real chord  through points $K,Q$ of~$\C$.  Every plane through a real chord is either a $2_\C$-plane or a $3_\C$-plane. Each of the $q-1$ points $R$ of $\C\setminus\{K,Q\}$ gives rise to the $3_\C$-plane through $K,Q,R$. Therefore, the number of
 $3_\C$-planes  through a real chord is equal to $q-1$. In total, we have $q+1$ planes through a line in $\PG(3,q)$. Thus, the number of $2_\C$-planes  through a real chord is $q+1-(q-1)=2$.
\end{proof}
\begin{proposition}\label{prop4_nd}
  For all $q$, we have
\begin{align}\label{eq4_ndsum}
&n_{0}^\Sigma=\frac{q(q^2-1)}{3},~ n_{1}^\Sigma=\frac{q^3+q+2}{2},~ n_{2}^\Sigma=q(q+1),~ n_{3}^\Sigma=\frac{q(q^2-1)}{6}.
\end{align}
\end{proposition}
\begin{proof}
By Theorem \ref{th2_HirsCor4,5}(B(i)), $n_{0}^\Sigma=\#\N_{5}$, $n_{1}^\Sigma=\#\N_{1}+\#\N_4$, $n_{2}^\Sigma=\#\N_{2}$, $n_{3}^\Sigma=\#\N_{3}$.
 \end{proof}

\begin{proposition}\label{prop4_Sd} The following holds:

\textbf{\emph{(i)}} Let $q\not\equiv0\pmod3$ and $q\equiv \xi\pmod 3$. Then for $\xi\ne0$ we have
    \begin{align*}
  n_{d,\text{\emph{T}}}^{(\xi)}+\frac{q-1}{3}n_{d,0_\Gamma}^{(\xi)}+\frac{q-1}{2}n_{d,1_\Gamma}^{(\xi)}+
   \frac{q-1}{6}n_{d,3_\Gamma}^{(\xi)}=\left\{\begin{array}{@{}l@{\,}c@{}}
   \frac{1}{3}(q^3-1)&\text{if }\,d=0 \smallskip\\
   \frac{1}{2}(q^3+q+2)&\text{if }\,d=1 \smallskip\\
   q^2+q-1&\text{if }\,d=2\smallskip \\
   \frac{1}{6}(q-1)^2(q+2)&\text{if }\,d=3
   \end{array}
   \right..
   \end{align*}

\textbf{\emph{(ii)}} Let $q\equiv0\pmod3$. Then
    \begin{align*}
  &(q-1)n_{d,\text{\emph{TO}}}^{(0)}+n_{d,q+1_\Gamma}^{(0)}+\frac{q(q-1)}{2}n_{d,\text{\emph{RC}}}^{(0)}+
   \frac{q(q-1)}{2}n_{d,\text{\emph{IC}}}^{(0)}=\left\{\begin{array}{@{}l@{\,}c@{}}
   \frac{1}{3}q(q^3-1)&\text{if }\,d=0 \smallskip\\
   \frac{1}{2}q(q^3+q+2)&\text{if }\,d=1 \smallskip\\
   q(q^2+q-1)&\text{if }\,d=2\smallskip \\
   \frac{1}{6}q(q-1)^2(q+2)&\text{if }\,d=3
   \end{array}
   \right..
   \end{align*}
\end{proposition}
\begin{proof}
Every $d_\C$-plane contains $q^2+q+1-d$ points outside $\C$. Therefore,
\begin{align*}
&\textbf{(i)} ~~\#\M_2n_{d,\text{T}}^{(\xi)}+\#\M_5n_{d,0_\Gamma}^{(\xi)}+\#\M_4n_{d,1_\Gamma}^{(\xi)}+
    \#\M_{3}n_{d,3_\Gamma}^{(\xi)}=n_{d}^\Sigma(q^2+q+1-d).\displaybreak[3]\\
&   \textbf{(ii)} ~~\#\M_3n_{d,\text{TO}}^{(0)}+\#\M_2n_{d,q+1_\Gamma}^{(0)}+\#\M_4n_{d,\text{RC}}^{(0)}+
    \#\M_5n_{d,\text{IC}}^{(0)}=n_{d}^\Sigma(q^2+q+1-d).
\end{align*}
 Now,  we use the values of $\#\M_j$ and $n_{d}^\Sigma$ from \eqref{eq2_point_orbits_gen}, \eqref{eq2_point_orbits_j=0}, and \eqref{eq4_ndsum}.
 \end{proof}

\begin{proposition}\label{prop4_q2+q+1}
  Let $q\equiv \xi\pmod3$. Then
  \begin{align*}
  &\sum_{d=0}^3n_{d,\text{\emph{T}}}^{(\xi)}=\sum_{d=0}^3n_{d,0_\Gamma}^{(\xi)}=\sum_{d=0}^3n_{d,1_\Gamma}^{(\xi)}=
  \sum_{d=0}^3n_{d,3_\Gamma}^{(\xi)}=q^2+q+1,~\xi\ne0;\displaybreak[3]\\
  &\sum_{d=0}^3 n_{d,\text{\emph{TO}}}^{(0)}=\sum_{d=0}^3 n_{d,q+1_\Gamma}^{(0)}=\sum_{d=0}^3 n_{d,\text{\emph{RC}}}^{(0)}=\sum_{d=0}^3 n_{d,\text{\emph{IC}}}^{(0)}=q^2+q+1.
  \end{align*}
\end{proposition}
\begin{proof}
 There are $q^2+q+1$ planes through every point of $\PG(3,q)$.
 \end{proof}

\begin{lem}\label{lem4_n2+3n3}
For all $q$, for a point $A$ off $\C$,
 \begin{align*}
  n_{2}(A)+3n_3(A)=\left\{\begin{array}{ll}
                     \binom{q+1}{2}&\text{if $A$ does not lie on any real chord}\smallskip \\
                     \frac{q^2+3q}{2}& \text{if $A$ lies on a real chord}
                   \end{array}\right..
 \end{align*}
\end{lem}
\begin{proof} Suppose $A$ does not lie on any real chord. There are $\binom{\#\C}{2}=\binom{q+1}{2}$ real chords, see Theorem \ref{th2_HirsCor4,5}(C(i)). Every chord together with $A$  defines a plane which is either a $2_\C$-plane or a $3_\C$-plane. All the $2_\C$-planes are distinct whereas every $3_\C$-plane contains 3 real chords and is repeated 3 times.

Let $A$ lie on a real chord. Let $S(A)$ be the set of $\binom{q+1}{2}-1$ real chords not containing~$A$. For $d=2,3$, let $n_{d}^*(A)$ be the number of $d_\C$-planes through $A$ and a chord of $S(A)$. Every such $3_\C$-plane contains 3 real chords of $S(A)$ and is repeated 3 times while all the $2_\C$-planes are distinct.

Denote by $\R\CC$ the real chord containing $A$.
 By Lemma \ref{lem4_3 2}, in total there are $q-1$ $3_\C$-planes  and two $2_\C$-planes through $\R\CC$. All these planes contain $A$ and they do not contain any chord from  $S(A)$.  Therefore, $  n_{3}(A) =n_{3}^*(A)+q-1,~n_{2}(A)=n_{2}^*(A)+2.$
 Each of the $q-1$ $3_\C$-planes  through $\R\CC$  contains 2 real chords of $S(A)$. Thus,
 \begin{align*}
   3n_{3}^*(A)+2(q-1)+n_{2}^*(A)=\binom{q+1}{2}-1
 \end{align*}
 whence the assertion follows.
 \end{proof}

\begin{corollary}\label{cor4_1_3}
 The following holds:
 \begin{align}\label{eq4_2_3i}
& n_{2,\text{\emph{T}}}^{(1)}+3n_{3,\text{\emph{T}}}^{(1)}= n_{2,\text{\emph{T}}}^{(-1)}+3n_{3,\text{\emph{T}}}^{(-1)}=n_{2,1_\Gamma}^{(1)}+3n_{3,1_\Gamma}^{(1)}=
 n_{2,0_\Gamma}^{(-1)}+3n_{3,0_\Gamma}^{(-1)}=n_{2,3_\Gamma}^{(-1)}\\
 &+3n_{3,3_\Gamma}^{(-1)}=n_{2,\text{\emph{TO}}}^{(0)}+3n_{3,\text{\emph{TO}}}^{(0)}= n_{2,q+1_\Gamma}^{(0)}+3n_{3,q+1_\Gamma}^{(0)}= n_{2,\text{\emph{IC}}}^{(0)}+3n_{3,\text{\emph{IC}}}^{(0)}=\binom{q+1}{2}.\notag\displaybreak[3]\\
& n_{2,0_\Gamma}^{(1)}+3n_{3,0_\Gamma}^{(1)}=n_{2,3_\Gamma}^{(1)}+3n_{3,3_\Gamma}^{(1)}=
 n_{2,1_\Gamma}^{(-1)}+3n_{3,1_\Gamma}^{(-1)}= n_{2,\text{\emph{RC}}}^{(0)}+3 n_{3,\text{\emph{RC}}}^{(0)}=\frac{q^2+3q}{2}\,.\label{eq4_2_3ii}
 \end{align}
\end{corollary}
\begin{proof} Due to Theorem \ref{th2_HirsCor4,5}(B(ii)),(B(iii)), \eqref{eq4_2_3i} holds for  points off $\C$ not on a real chord whereas \eqref{eq4_2_3ii} concerns points lying on a real chord.
 \end{proof}

\begin{lem}\label{lem4_n1+2n2+3n3}  For all $q$, for a point $A$ off $\C$ the following holds:
 \begin{align*}
  n_{1}(A)+2n_{2}(A)+3n_3(A)=(q+1)^2.
 \end{align*}
\end{lem}
\begin{proof} We consider the line $\overline{AP}_i$ through points $A\notin\C$ and $P_i\in\C$, $i\in\{1,2,\ldots,q+1\}$. Each of the $q+1$ planes through $\overline{AP}_i$ is a $d_\C$-plane with $d\in\{1,2,3\}$. Let $n_d(P_i)$ be the number of $d_\C$-planes through $\overline{AP}_i$. Clearly, $n_1(P_i)+n_2(P_i)+n_3(P_i)=q+1$. Moreover,
\begin{align*}
    & n_{1}(A)+2n_{2}(A)+3n_3(A)=\sum_{i=1}^{q+1}\left(n_1(P_i)+n_2(P_i)+n_3(P_i)\right)=\sum_{i=1}^{q+1}(q+1)=(q+1)^2.
\end{align*}
Here we take into account that in the sum $\sum_{i=1}^{q+1}\left(n_1(P_i)+n_2(P_i)+n_3(P_i)\right)$ every  $d_\C$-plane appears $d$ times.
 \end{proof}

\begin{corollary}\label{cor4_n1+2n2+3n3}
  For all $q$, the following holds:
\begin{align*}
&n_{1,\text{\emph{T}}}^{(\xi)}+2n_{2,\text{\emph{T}}}^{(\xi)}+3n_{3,\text{\emph{T}}}^{(\xi)}
=n_{1,\mu_{\Gamma}}^{(\xi)}+2n_{2,\mu_\Gamma}^{(\xi)}+3n_{3,\mu_\Gamma}^{(\xi)}=(q+1)^2,~
\mu_\Gamma=0,1,3,~\xi\ne0;\displaybreak[3]\\
&n_{1,\text{\emph{TO}}}^{(0)}+2n_{2,\text{\emph{TO}}}^{(0)}+3n_{3,\text{\emph{TO}}}^{(0)}= n_{1,q+1_\Gamma}^{(0)}+2n_{2,q+1_\Gamma}^{(0)}+3n_{3,q+1_\Gamma}^{(0)}\displaybreak[3]\\
&= n_{1,\text{\emph{RC}}}^{(0)}+ 2n_{2,\text{\emph{RC}}}^{(0)}+3n_{3,\text{\emph{RC}}}^{(0)}
= n_{1,\text{\emph{IC}}}^{(0)}+ 2n_{2,\text{\emph{IC}}}^{(0)}+3n_{3,\text{\emph{IC}}}^{(0)}=(q+1)^2.\notag
\end{align*}
 \end{corollary}

\begin{lem}\label{lem4_imag_chord}
 All $d_\C$-planes with $d=0,2,3$ and all osculating planes contain no imaginary chord.
All $q+1$ planes through an imaginary chord are $1_\C\setminus\Gamma$-planes.
\end{lem}
\begin{proof}
    Any $2_\C$-plane and $3_\C$-plane contains a real chord. An osculating plane contains a tangent. If a $2_\C$,- or a $3_\C$-, or a $\Gamma$-plane contains an imaginary chord then it intersects the real chord or the tangent, contradiction, see Theorem \ref{th2_HirsCor4,5}(C(ii)). Thus, we have a $1_\C\setminus\Gamma$-plane through an imaginary chord and any point of $\C$. In total, there are $\#\C=q+1$ such $1_\C\setminus\Gamma$-planes for every imaginary chord.
 \end{proof}

The following lemma is obvious.
\begin{lem}\label{lem4_pointinplane}
In $\PG(3,q)$, let $\N$ and $\M$ be, respectively, an orbit of planes and an orbit of points under some group $G$ of projectivities.
\begin{description}
  \item[(i)] The
number of planes from $\N$ through a point of $\M$ is the same for all points of~$\M$.
  \item[(ii)]
  The number of points from $\M$ in a plane of $\N$ is the same for all planes of $\N$.
\end{description}
\end{lem}
\begin{proof}
\begin{description}
  \item[(i)]
Consider points $P$ and $Q$ of~$\M$. Denote by $\pi$ a plane of $\N$. Let $S(P)$ and $S(Q)$ be subsets of $\N$ such that $S(P)=\{\pi\in\N|P\in\pi\}$, $S(Q)=\{\pi\in\N|Q\in\pi\}$. There exists $\varphi\in G$  such that  $Q=\varphi(P)$. Clearly, $\varphi$ embeds $S(P)$ in $S(Q)$, i.e. $\varphi(S(P))\subseteq S(Q)$ and $\#S(P)\le\#S(Q)$. In the same way, $\varphi^{-1}$ embeds $S(Q)$ in $S(P)$, i.e.  $\#S(Q)\le\#S(P)$. Thus,  $\#S(Q)=\#S(P)$.
  \item[(ii)] The proof is similar to part (i).
\end{description}
 \end{proof}

\section{The number $r_{ij}$ of distinct planes through distinct points of $\PG(3,q)$}\label{sec_exac_res}

In this section we obtain all values $r_{ij}$, $i,j=1,\ldots,5$.

\begin{thm}\label{th5_cubic} The following holds:
\begin{align*}
    n_{0,\C}=0,~ n_{1,\C}=\frac{q^2-q+2}{2},~ n_{2,\C}=2q,~ n_{3,\C}=\frac{q^2-q}{2}.
\end{align*}
\end{thm}

\begin{proof}By  definition, $n_{0,\C}=0$.
Obviously,  $n_{1,\C}=\frac{n_{1}^\Sigma}{\#\C}$, see \eqref{eq4_ndsum}.

We consider a point $A\in\C$. There are $q$ real chords through $A$. By Lemma~\ref{lem4_3 2}, we have two $2_\C$-planes through every such chord.  Finally, every  pair of points of $\C\setminus\{A\}$ generates a $3_\C$-plane through~$A$.
 \end{proof}

\begin{thm}\label{th5_nd1_1}
 The following holds:
 \begin{align*}
&n_{0,1_\Gamma}^{(1)}=n_{0,q+1_\Gamma}^{(0)}=n_{0,\text{\emph{IC}}}^{(0)}=\frac{q^2-q}{3},~
n_{1,1_\Gamma}^{(1)}=n_{1,q+1_\Gamma}^{(0)}=n_{1,\text{\emph{IC}}}^{(0)}=\frac{q^2+q+2}{2},\displaybreak[3]\\
&n_{2,1_\Gamma}^{(1)}=n_{2,q+1_\Gamma}^{(0)}=n_{2,\text{\emph{IC}}}^{(0)}=q,~n_{3,1_\Gamma}^{(1)}=n_{3,q+1_\Gamma}^{(0)}=
n_{3,\text{\emph{IC}}}^{(0)}=\frac{q^2-q}{6}.
 \end{align*}
\end{thm}
\begin{proof}
By Theorem \ref{th2_HirsCor4,5}(B(ii)), for $q\equiv1\pmod3$,  $1_\Gamma$-points  are points on imaginary chords. We take  an imaginary chord $\I\CC$.
Clearly, $\#\I\CC=q+1$. By Lemma~\ref{lem4_imag_chord},
all $n_0^\Sigma$ $0_\C$-planes intersect $\I\CC$.
By Theorem~\ref{th2_HirsCor4,5}(B(ii)), for $q\not\equiv0\pmod3$, all $1_\Gamma$-points belong to the same orbit of the group~$G_q$.
Therefore, the number of $d_\C$-planes intersecting every $1_\Gamma$-point is the same. Thus, see also Proposition \ref{prop4_nd}, we have
\begin{align*}
    n_{0,1_\Gamma}^{(1)}=\frac{n_0^\Sigma}{\#\I\CC}=
     \frac{q^2-q}{3}\,.
\end{align*}
By Proposition \ref{prop4_q2+q+1},
\begin{align*}
   \sum_{d=1}^3n_{d,1_\Gamma}^{(1)}=q^2+q+1- \frac{q^2-q}{3}\,.
\end{align*}
This equation  together with Corollaries \ref{cor4_1_3} and  \ref{cor4_n1+2n2+3n3} yields $n_{d,1_\Gamma}^{(1)}$, $d=1,2,3$.

A similar argument holds for $n_{d,\text{IC}}^{(0)}$ and  for $n_{d,q+1_\Gamma}^{(0)}$ (together with  Remark \ref{rem2_q+1 osc}).
 \end{proof}

\begin{thm}\label{th5 tangent}
 Let $q\not\equiv0\pmod3$. Then
\begin{align*}
n_{0,\text{\emph{T}}}^{(\ne0)}=\frac{q^2-1}{3},~n_{1,\text{\emph{T}}}^{(\ne0)}=\frac{q^2-q+4}{2},~n_{2,\text{\emph{T}}}^{(\ne0)}=2q-1,
~n_{3,\text{\emph{T}}}^{(\ne0)}=\frac{q^2-3q+2}{6}.
\end{align*}
\end{thm}
\begin{proof}
We proceed as in Theorem \ref{th5_nd1_1}.

We consider a tangent line $\T$  to $\C$ at a point $Q\in\C$. We denote $\widehat{\T}=\T\setminus\{Q\}$.
Clearly, $\widehat{\T}$ consists of T-points and $\#\widehat{\T}=q$. All $n_0^\Sigma$ $0_\C$-planes intersect $\widehat{\T}$.
By Theorem~\ref{th2_HirsCor4,5}(B(ii)), for $q\not\equiv0\pmod3$, all T-points belong to the same orbit of the group $G_q$; the number of $d_\C$-planes intersecting every T-point is the same.
Therefore,
\begin{align*}
    n_{0,\text{T}}^{(\ne0)}=\frac{n_0^\Sigma}{\#\widehat{\T}}=
     \frac{q^2-1}{3}.
\end{align*}
By  Proposition \ref{prop4_q2+q+1} and  Corollaries \ref{cor4_1_3} and  \ref{cor4_n1+2n2+3n3}, the claim follows.
 \end{proof}

\begin{thm}\label{th5_RC}
 The following holds:
 \begin{align*}
&n_{0,1_\Gamma}^{(-1)}=n_{0,\text{\emph{RC}}}^{(0)}=\frac{q^2+q}{3},
~n_{1,1_\Gamma}^{(-1)}=n_{1,\text{\emph{RC}}}^{(0)}=\frac{q^2-q+2}{2},\displaybreak[3]\\  &n_{2,1_\Gamma}^{(-1)}=n_{2,\text{\emph{RC}}}^{(0)}=q,~n_{3,1_\Gamma}^{(-1)}=n_{3,\text{\emph{RC}}}^{(0)}=\frac{q^2+q}{6}\,.
 \end{align*}
\end{thm}
\begin{proof}
We proceed as in Theorems \ref{th5_nd1_1} and \ref{th5 tangent}.

By Theorem \ref{th2_HirsCor4,5}(B(ii)), for $q\equiv-1\pmod3$,  $1_\Gamma$-points  are points on real chords. We take  a real chord $\R\CC$ through points $Q,K$ of $\C$. We denote $\widehat{\R\CC}=\R\CC\setminus\{Q,K\}$.
Clearly, $\widehat{\R\CC}$ consists of $1_\Gamma$-points  and $\#\widehat{\R\CC}=q-1$. All $n_0^\Sigma$ $0_\C$-planes intersect $\widehat{\R\CC}$.
Also, by Theorem~\ref{th2_HirsCor4,5}(B(ii)), for $q\equiv-1\pmod3$, all $1_\Gamma$-points belong to the same orbit of the group $G_q$; the number of $d_\C$-planes intersecting every $1_\Gamma$-point is the same.
Therefore,
\begin{align*}
    n_{0,1_\Gamma}^{(-1)}=\frac{n_0^\Sigma}{\#\widehat{\R\CC}}=
     \frac{q^2+q}{3}.
\end{align*}
The claim follows using  Proposition \ref{prop4_q2+q+1} and  Corollaries \ref{cor4_1_3} and  \ref{cor4_n1+2n2+3n3}. The argument for $n_{d,\text{RC}}^{(0)}$ is the same.
 \end{proof}

\begin{lem}\label{lem5+3-product}
  Let $q\equiv1\pmod3$. Let $\mathbb{T}$ be the $\binom{q-1}{3}$-multiset of all possible products of three distinct elements of $\F_q^*$. Then in $\mathbb{T}$, cubes (resp. non-cubes) of $\F_q^*$ appear $m_c$ (resp. $m_{nc}$) times, where
   \begin{align*}
   m_c= \frac{q-1}{3}\cdot\frac{q^2-5q+10}{6},~m_{nc}= \frac{2(q-1)}{3}\cdot\frac{q^2-5q+4}{6}.
\end{align*}
\end{lem}
\begin{proof}Let $\alpha$ be a primitive element of $\F_q$. We partition $\F_q^*$ into three $\frac{q-1}{3}$-subsets with elements of the form $\alpha^{3v}$, $\alpha^{3v+1}$, and $\alpha^{3v+2}$, respectively. A product of three distinct elements of $\F_q^*$ is a cube if and only if all three elements belong to the same subset or to distinct subsets. So,
\begin{align*}
    3\binom{(q-1)/3}{3}+\left(\frac{q-1}{3}\right)^3=m_c.
\end{align*}
Finally, $m_{nc}=\binom{q-1}{3}-m_c$.
 \end{proof}

\begin{thm}\label{th5 n30n33(1)}
 Let $q\equiv1\pmod3$. Then
\begin{align*}
&n_{3,0_\Gamma}^{(1)}=\frac{q^2+q-2}{6},~ n_{3,3_\Gamma}^{(1)}=\frac{q^2+q+4}{6}\,.
\end{align*}
\end{thm}
\begin{proof}
We consider the real chord $\R\CC_{0,\infty}$ through $P(0)=\P(0,0,0,1)$ and $P(\infty)=\P(1,0,0,0)$. We denote $\widehat{\R\CC}_{0,\infty}=\R\CC_{0,\infty}\setminus\{P(0),P(\infty)\}$.  Points in $\widehat{\R\CC}_{0,\infty}$ have the form $(c,0,0,1)$, $c\in\F_q$. By \eqref{eq2_osc pl}, $\pi_\Gamma(t)=\boldsymbol{\pi}(1,-3t,3t^2,-t^3)$.  Therefore, in  $\widehat{\R\CC}_{0,\infty}$, we have $3_\Gamma$-points of the form $\P(a^3,0,0,1)$, $a\in\F_q$, and $0_\Gamma$-points of the form $\P(a^v,0,0,1)$, $a\in\F_q$, $v\not\equiv0\pmod3$. In $\widehat{\R\CC}_{0,\infty}$, the number of $3_\Gamma$-points and $0_\Gamma$-points is $\frac{q-1}{3}$ and $\frac{2(q-1)}{3}$, respectively.

By \eqref{eq2_plane3points}, a $3_\Gamma$-point $\P(a^3,0,0,1)$ and a $0_\Gamma$-point $\P(a^v,0,0,1)$ lie on the plane through three points $P(t_1)$, $P(t_2),P(t_3)$ of $\C$ if $a^3=t_1t_2t_3$ and $a^v=t_1t_2t_3$, respectively.
Now, by Lemma~\ref{lem5+3-product}, one sees that through $3_\Gamma$-points of $\widehat{\R\CC}_{0,\infty}$, in total, there are  $m_c$ $3_\C$-planes not containing the points $P(0),P(\infty)$. Also, by Lemma \ref{lem4_3 2}, through every  $3_\Gamma$-point of $\widehat{\R\CC}_{0,\infty}$ we have $q-1$  $3_\C$-planes containing $\R\CC_{0,\infty}$. Thus,  through  $3_\Gamma$-points on $\R\CC_{0,\infty}$ we have, in total, $m_c+ \frac{q-1}{3}(q-1)$ $3_\C$-planes. All $3_\Gamma$-points belong to the same orbit $\M_3$ under $G_q$. Therefore, the number of $3_\C$-planes through a $3_\Gamma$-point on $\R\CC_{0,\infty}$ is  equal to
\begin{align*}
  \left(m_c+ \frac{q-1}{3}(q-1)\right)\left(\frac{q-1}{3}\right)^{-1}=\frac{q^2+q+4}{6}
\end{align*}

Similarly, the number of $3_\C$-planes through a $0_\Gamma$-point on $\R\CC_{0,\infty}$ is
\begin{align*}
    \left(m_{nc}+ \frac{2(q-1)}{3}(q-1)\right)\left(\frac{2(q-1)}{3}\right)^{-1}=\frac{q^2+q-2}{6}.
\end{align*}

Finally, note that the number of intersecting $d_\C$-planes is the same for all points of an orbit under $G_q$.
 \end{proof}

\begin{thm}\label{th5 n20n23(1)}
 Let $q\equiv1\pmod3$. Then
\begin{align*}
&n_{0,0_\Gamma}^{(1)}=\frac{q^2+q+1}{3},~n_{1,0_\Gamma}^{(1)}=\frac{q^2-q}{2},~n_{2,0_\Gamma}^{(1)}=q+1;\displaybreak[3]\\
&n_{0,3_\Gamma}^{(1)}=\frac{q^2+q-2}{3},~n_{1,3_\Gamma}^{(1)}=\frac{q^2-q+6}{2},~n_{2,3_\Gamma}^{(1)}=q-2.
\end{align*}
\end{thm}
\begin{proof}
By Corollary \ref{cor4_1_3} and Theorem \ref{th5 n30n33(1)}, we obtain $n_{2,0_\Gamma}^{(1)}$ and $n_{2,3_\Gamma}^{(1)}$. Then by Corollary \ref{cor4_n1+2n2+3n3}  we get $n_{1,0_\Gamma}^{(1)}$ and $n_{1,3_\Gamma}^{(1)}$. Finally, we use Proposition \ref{prop4_q2+q+1} for $n_{0,0_\Gamma}^{(1)}$ and $n_{0,3_\Gamma}^{(1)}$.
 \end{proof}

\begin{thm}\label{th5_TO}
    Let $q\equiv0\pmod3$. Then
    \begin{align*}
n_{0,\text{\emph{TO}}}^{(0)}=\frac{q^2}{3},~n_{1,\text{\emph{TO}}}^{(0)}=\frac{q^2-q+2}{2},
~n_{2,\text{\emph{TO}}}^{(0)}=2q,
~n_{3,\text{\emph{TO}}}^{(0)}=\frac{q^2-3q}{6}.
    \end{align*}
\end{thm}
\begin{proof}
We consider a tangent line $\T$  to $\C$ at a point $Q\in\C$. Let $S$ be the $(q+1)_\Gamma$-point on $\T$. We denote $\widetilde{\T}=\T\setminus\{Q,S\}$.
Clearly, $\widetilde{\T}$ consists of TO-points and $\#\widetilde{\T}=q-1$. All $n_0^\Sigma$ $0_\C$-planes intersect $\T\setminus\{Q\}$. Therefore, the total number of $0_\C$-planes intersecting $\widetilde{\T}$ is $n_0^\Sigma-n_{0,q+1_\Gamma}^{(0)}$ where we subtract $0_\C$-planes through $S$.
By Theorem~\ref{th2_HirsCor4,5}(B(ii)), for $q\equiv0\pmod3$, all TO-points belong to the same orbit of the group $G_q$; the number of $d_\C$-planes intersecting every TO-point is the same.
Therefore, see also Theorem \ref{th5_nd1_1},
\begin{align*}
    n_{0,\text{TO}}^{(0)}=\frac{n_0^\Sigma-n_{0,q+1_\Gamma}^{(0)}}{\#\widetilde{\T}\vphantom{H^{H^H}}}=
     \frac{q^2}{3}.
\end{align*}
The claim follows from  Proposition \ref{prop4_q2+q+1} and  Corollaries \ref{cor4_1_3} and  \ref{cor4_n1+2n2+3n3}.
 \end{proof}

\begin{proposition}\label{prop5_q=2mod3_mu=03}
 Let $q\equiv-1\pmod3$. Then
 \begin{align*}
 &2n_{0,0_\Gamma}^{(-1)}+n_{0,3_\Gamma}^{(-1)}=q^2-q,~2n_{1,0_\Gamma}^{(-1)}+n_{1,3_\Gamma}^{(-1)}=
 \frac{3(q^2+q+2)}{2},\displaybreak[3]\\
 &2n_{2,0_\Gamma}^{(-1)}+n_{2,3_\Gamma}^{(-1)}=3q,~2n_{3,0_\Gamma}^{(-1)}+n_{3,3_\Gamma}^{(-1)}=\frac{q^2-q}{2}.\notag
 \end{align*}
\end{proposition}
\begin{proof}
By Theorem \ref{th2_HirsCor4,5}(B(ii)), for $\mu_\Gamma=0,3$, all $\mu_\Gamma$-points belong to the same orbit under~$G_q$.
By Theorem \ref{th2_HirsCor4,5}(B(ii)), for $q\equiv-1\pmod3$, we have that $0_\Gamma$-points and $3_\Gamma$-points are points on imaginary chords. By Lemma \ref{lem4_imag_chord}, for $d=0,2,3$, all  $n_d^\Sigma$ $d_\C$-planes intersect all $\binom{q}{2}$ imaginary chords.  Thus, the total number of intersections of imaginary chords with $d_\C$-planes is $\binom{q}{2}n_d^\Sigma$. So,
\begin{align*}
   &  \#\M_5n_{d,0_\Gamma}^{(-1)}+\#\M_{3}n_{d,3_\Gamma}^{(-1)}=
  \binom{q}{2}n_d^\Sigma,~d=0,2,3.\end{align*}
The assertions for $d=0,2,3 $ follow from \eqref{eq2_point_orbits_gen}, \eqref{eq4_ndsum}.

Finally, by Proposition \ref{prop4_q2+q+1}, we obtain
  \begin{align*}
  2\sum_{d=0}^3n_{d,0_\Gamma}^{(-1)}+  \sum_{d=0}^3n_{d,3_\Gamma}^{(-1)}=3(q^2+q+1).
  \end{align*}
 \end{proof}

\begin{lem}\label{lem5_a^2+a+1}
Let $q\equiv-1\pmod3$ be odd. Let $f(a)=a^2+a+1$.
 Let $V=\{a\in\F_q | f(a) \text{ is a square in }\F_q\}$. Then $\#V=\frac{q-1}{2}.$
\end{lem}
\begin{proof}
By \cite[Theorem 5.18]{Lidl_Nied}, $\sum_{a\in\F_q}\eta(f(a))=-\eta(1)=-1$ where $\eta$ is the quad\-ratic character of $\F_q$. Also, $f(a) \neq 0, \forall a \in \F_q$. So, $\#V-(q-\#V)=-1$.
 \end{proof}

\begin{lem}\label{lem5_W}
   Let $q\equiv-1\pmod3$. Then the point $W=\P(0,1,-1,0)$ off $\C$ lies on three osculating planes. Moreover, the number of $3_\C$-planes through $W$ is equal to $(q^2-q+4)/6$.
\end{lem}
\begin{proof}
By \eqref{eq2_osc pl}, $W$ belongs to $\pi_\Gamma(t)$ with $-3t-3t^2=0$ whence $t=0,1.$ Also, by \eqref{eq2_osc pl_inf}, $W$ lies on $\pi_\Gamma(\infty)$.

\textbf{(1)} The $3_\C$-plane $\pi'$ through  points $P(t_1),P(t_2),P(\infty)$ of $\C$ has the form
\begin{align*}
   \pi'= \boldsymbol{\pi}(0,-1,t_1+t_2,-t_1t_2)\supset\{P(t_1),P(t_2),P(\infty)\}.
\end{align*}
This means that  $W$ belongs to $\pi'$ if $  -1-t_1-t_2=0.$ So,
under the condition $t_1\ne t_2$, there are $n'$ distinct $3_\C$-planes $\pi'$ through $W$ where
\begin{align*}
   n'=\left\{\begin{array}{ccc}
\frac{q}{2} & \text{if}& q\text{ even}\smallskip \\
\frac{q-1}{2} & \text{if} & q\text{ odd}
             \end{array}
   \right..
\end{align*}

\textbf{(2)} By \eqref{eq2_plane3points}, the $3_\C$-plane $\pi''$ through  points $P(t_1),P(t_2),P(t_3)$ with $t_i\ne\infty$, $i=1,2,3$, contains $W$ under the condition
\begin{align}\label{eq5_pi''}
  (t_1+t_2+t_3)+(t_1t_2+t_1t_3+t_2t_3)=0,~t_i\in\F_q,~t_i\ne t_j,~i,j\in\{1,2,3\}.
\end{align}

We now compute the number $n''$ of distinct triples $t_1,t_2,t_3$ satisfying \eqref{eq5_pi''}.

\textbf{(2.1)} Let $q$ be even, i.e. $q=2^{2v+1}\equiv-1\pmod3$.

In this case, by \eqref{eq5_pi''}, we have
\begin{align}\label{eq5_t3<-t1,t2}
   t_3=\frac{t_1+t_2+t_1t_2}{1+t_1+t_2}.
\end{align}
We fix $t_1\in\F_q$.  By \eqref{eq5_pi''} and \eqref{eq5_t3<-t1,t2}, there are the following restrictions on $t_2$:\\
(a) $t_2\ne t_1$;\\
(b) $t_2\ne t_1+1$ otherwise $1+t_1 +t_2=0$;\\
(c) $t_2\ne t_3$ whence $t_2(1+t_1+t_2)\ne t_1+t_2+t_1t_2$ and $t_2\ne \sqrt{t_1}$;\\
(d) $t_1\ne t_3$ whence $t_1(1+t_1+t_2)\ne t_1+t_2+t_1t_2$ and $t_2\ne t_1^2$.

Suppose (a) and (c) or (a) and (d) coincide, i.e. $t_1=t_1^2$ or $t_1=\sqrt{t_1}$. This  implies $t_1=0,1$.

Suppose (b) and (c) or (b) and (d) coincide, i.e. $t_1 +1=t_1^2$ or $t_1 +1=\sqrt{t_1}$. This yields $t_1^2 +t_1 +1=0$. As $q=2^{2v+1}$,  the trace $\text{Tr}_{\F_{q}}(1)\ne0$ \cite[Corollary 3.79]{Lidl_Nied}, a contradiction.

Finally, if (c) and (d) coincide then $\sqrt{t_1}=t_1^2$, $t_1=t_1^4$ and therefore $t_1=0,1$.

Thus, for $t_1\in\F_q$, $t_1\ne0,1$,  (a)--(d)  are distinct. Here we have $q-2$ possibilities for  $t_1$ and  $q-4$ possibilities for $t_2$ for every $t_1$.   Also, there are  $q-2$ possibilities of $t_2$  if $t_1=0,1$.

The number of distinct triples $t_1,t_2,t_3$ satisfying \eqref{eq5_pi''} is therefore
$(q-2)(q-4)+2(q-2) = q^2-4q+4$. Because of symmetry, each plane is generated by 6 triples, so
$n'' = (q^2-4q+4)/6$.

Now $n'+n''$ gives the needed result for even $q$.

\textbf{(2.2)} Let $q$ be odd, i.e. $q=p^{2v+1}$, $p>3$ prime, $p\equiv-1\pmod3$.

First we count the number of triples satisfying
\begin{align}\label{eq5_3t}
  (t_1+t_2+t_3)+(t_1t_2+t_1t_3+t_2t_3)=0,~t_i\in\F_q,
\end{align}
\emph{without the condition}  $~t_i\ne t_j,~i,j\in\{1,2,3\}.$

Relation \eqref{eq5_3t} can be rewritten as the set of $q$ conditions
\begin{align}\label{eq5_system}
\left\{
\begin{array}{c}
 t_1+t_2+t_3=k \\
 t_1t_2+t_1t_3+t_2t_3=-k \\
\end{array}
\right.
\end{align}
where  $k\in\F_q$.\\
 The triples satisfying \eqref{eq5_system} can be seen as the affine coordinates of the points of the 3-dimensional affine space $\mathrm{AG}(3,q)$ belonging to a plane conic defined by
 \begin{align}\label{eq5_conic}
\left\{
\begin{array}{c}
 t_1+t_2+t_3=k \\
  t^2_2+t^2_3+t_2t_3-k t_2-k t_3-k=0
\end{array}
\right..
\end{align}
For $k = 0$ and $k=-3$, the conic is degenerate and, as $\sqrt{-3}$ is not a square in $\mathbb{F}_q$, the unique triples satisfying \eqref{eq5_conic} are $(0,0,0)$ and $(-1,-1,-1)$.

For each $k\in \mathbb{F}_{q}\setminus \{0,-3\}$, there are exactly $q+1$ triples $(t_1,t_2,t_3)$ satisfying~\eqref{eq5_conic}.

Therefore $2+(q-2)(q+1)= q(q-1)$ triples satisfy \eqref{eq5_3t}.\\
To count the triples satisfying \eqref{eq5_pi''}, we exclude the triples satisfying \eqref{eq5_3t} having at least two equal elements.

\textbf{(2.2.1)}  $t_1=t_2=t_3$.\\
Equation \eqref{eq5_3t} reads $3t_1 + 3t_1^2=0$, so $t_1= 0,-1$.

\textbf{(2.2.2)}  $t_i=t_j \neq t_k,~i,j,k\in\{1,2,3\}.$\\
Equation \eqref{eq5_3t} reads
\begin{align}\label{eq5_2t}
  t^2_i + 2 (t_k+1 ) t_i+t_k=0.
\end{align}
Discriminant of  \eqref{eq5_2t} is $4(t_k^2+t_k+1)$.   Let $V=\{t_k\in\F_q | t_k^2+t_k+1$ is a square in $\F_q\}$. By Lemma \ref{lem5_a^2+a+1}, $\#V=\frac{q-1}{2}.$
As $q\equiv-1\pmod3$, $q$ odd, by \cite[Chapter 1]{Hirs_PGFF} $t^2_k+t_k+1\ne0, ~\forall t_k \in \F_q$. Then $\forall t_k \in V$ we obtain two distinct values of $t_i$. On the other hand, when $t_k = 0,-1$, one of the values of $t_i$ we obtain is equal to $t_k$. Therefore the number of triples satisfying~\eqref{eq5_3t} such that $t_i=t_j \neq t_k,~i,j,k\in\{1,2,3\}$, is $3(2( \frac{q-1}{2}-2)+2)= 3(q-3)$.

So, the number of distinct triples $t_1,t_2,t_3$ satisfying \eqref{eq5_pi''} is
$q(q-1) -2-3(q-3) = q^2-4q+7$. Because of symmetry, each plane is generated by 6 triples, so
$n'' =(q^2-4q+7)/6$.
Now $n'+n''$ gives the needed result for odd $q$.
 \end{proof}

\begin{thm}\label{th5 n20n23(2)}
 Let $q\equiv-1\pmod3$. Then
\begin{align*}
&n_{0,0_\Gamma}^{(-1)}=\frac{q^2-q+1}{3},~n_{1,0_\Gamma}^{(-1)}=\frac{q^2+q}{2},
~n_{2,0_\Gamma}^{(-1)}=q+1,~n_{3,0_\Gamma}^{(-1)}=\frac{q^2-q-2}{6};\displaybreak[3]\\
&n_{0,3_\Gamma}^{(-1)}=\frac{q^2-q-2}{3},~n_{1,3_\Gamma}^{(-1)}=\frac{q^2+q+6}{2},
~n_{2,3_\Gamma}^{(-1)}=q-2,~n_{3,3_\Gamma}^{(-1)}=\frac{q^2-q+4}{6}.
\end{align*}
\end{thm}

\begin{proof}
As all points of the orbit $\M_{3}$ have the same number of intersecting $d_\C$-planes, we have by Lemma \ref{lem5_W} that $n_{3,3_\Gamma}^{(-1)}=\frac{q^2-q+4}{6}$. Then we obtain the value $n_{3,0_\Gamma}^{(-1)}$ by Proposition~\ref{prop5_q=2mod3_mu=03}. By Lemma \ref{lem4_n2+3n3} and Corollary \ref{cor4_1_3}, see \eqref{eq4_2_3i}, we obtain $n_{2,0_\Gamma}^{(-1)}$ and $n_{2,3_\Gamma}^{(-1)}$. Then by Lemma \ref{lem4_n1+2n2+3n3} and Corollary \ref{cor4_n1+2n2+3n3}, we get $n_{1,0_\Gamma}^{(-1)}$ and $n_{1,3_\Gamma}^{(-1)}$. Finally, we use Proposition~\ref{prop4_q2+q+1} for $n_{0,0_\Gamma}^{(-1)}$ and $n_{0,3_\Gamma}^{(-1)}$.
 \end{proof}

\begin{thm}\label{th5_points&1-planes}
   For $q\equiv\xi\pmod3$, the following holds:
   \begin{description}
     \item[(i)]$\xi=-1,1$.
   \begin{align*}
    &r_{11}=r_{14}=1,~r_{12}=2,~r_{13}=3,~r_{15}=0,\displaybreak[3]\\
    &r_{41}=r_{42}=\frac{1}{2}(q^2-q),~r_{43}=r_{45}=\frac{1}{2}(q^2-\xi q),~r_{44}=\frac{1}{2}(q^2+\xi q).
   \end{align*}
     \item[(ii)]$\xi=0$.
   \begin{align*}
    &r_{11}=r_{13}=r_{14}=r_{15}=1,~r_{12}=q+1,~\displaybreak[3]\\
    &r_{41}=r_{42}=r_{43}=r_{44}=\frac{1}{2}(q^2-q),~r_{45}=\frac{1}{2}(q^2+q).
   \end{align*}
   \end{description}
\end{thm}

\begin{proof}
   \begin{description}
     \item[(i)]
     By definition, $r_{11}=r_{14}=1,~r_{13}=3,~r_{15}=0$.

     We consider a tangent $\T$  to $\C$ at a point $Q$ of $\C$. We denote $\widehat{\T}=\T\setminus\{Q\}$.
Clearly, $\widehat{\T}$ consists of T-points and lies in a $\Gamma$-plane. The rest $q$ osculating planes intersect~$\widehat{\T}$. As all $q$ points of $\widehat{\T}$ belong to the same orbit under $G_q$, every point corresponds to $\frac{q}{\#\widehat{\T}^{\vphantom{H^H}}}=\frac{q}{q}=1$ intersection. Thus, $r_{12}=2$.

We note, see Table \ref{tab1} and Notation \ref{not1b}, that $r_{41}=n_{1,\C}-r_{11}$,~$r_{42}=n_{1,\text{T}}^{(\ne0)}-r_{12}$,~
$r_{43}=n_{1,3_\Gamma}^{(\xi)}-r_{13}$,~$r_{44}=n_{1,1_\Gamma}^{(\xi)}-r_{14}$,~
$r_{45}=n_{1,0_\Gamma}^{(\xi)}-r_{15}$. Finally, we take the values $n_{1,\C},n_{1,\text{T}}^{(\ne0)},n_{1,\mu_\Gamma}^{(\xi)}$ from Theorems \ref{th5_cubic}--\ref{th5_RC}, \ref{th5 n20n23(1)}, and \ref{th5 n20n23(2)}.
     \item[(ii)]
     By definition, $r_{11}=1$, $r_{12}=q+1$.

     We consider a tangent line $\T$  to $\C$ at a point $Q\in\C$. Let $K$ be the $(q+1)_\Gamma$-point in~$\T$. We denote $\widehat{\T}=\T\setminus\{Q,K\}$.
Clearly, $\widehat{\T}$ consists of OT-points, see Remark~\ref{rem2_q+1 osc}. All $\Gamma$-planes form a pencil of planes; their common line passes through~$K$. Therefore,  no  $\Gamma$-plane intersects $\widehat{\T}$. On the other hand, $\widehat{\T}$ lies in the $\Gamma$-plane through $Q.$ So, $r_{13}=1$.

     We consider  a real chord $\R\CC$ through points $Q,K$ of $\C$. We denote $\widehat{\R\CC}=\R\CC\setminus\{Q,K\}$. Apart from the osculating planes through $Q$ and $K$, all the other $q-1$ such planes intersect $\widehat{\R\CC}$.
All $q-1$ points of $\widehat{\R\CC}$ belong to the same orbit under~$G_q$. Therefore, the number of the osculating planes through every point of $\widehat{\R\CC}$  is the same and $r_{14}=\frac{q-1}{q-1}=1$.

     We take  an imaginary chord $\I\CC$.
 By Lemma~\ref{lem4_imag_chord},
all $q+1$ osculating planes  intersect~$\I\CC$. As all $q+1$ points of $\I\CC$ belong to the same orbit under $G_q$, the number of the osculating planes through every point of $\I\CC$  is the same and $r_{15}=\frac{q+1}{q+1}=1$.

We note, see Table \ref{tab2} and Notation \ref{not1b}, that $r_{41}=n_{1,\C}-r_{11}$,~$r_{42}=n_{1,q+1_\Gamma}^{(0)}-r_{12}$,~
$r_{43}=n_{1,\text{TO}}^{(0)}-r_{13}$,~$r_{44}=n_{1,\text{RC}}^{(0)}-r_{14}$,~
$r_{45}=n_{1,\text{IC}}^{(0)}-r_{15}$. Finally, Theorems \ref{th5_cubic}, \ref{th5_nd1_1}, \ref{th5_RC}, and \ref{th5_TO} provide $n_{1,\C},n_{1,q+1_\Gamma}^{(0)},n_{1,\text{TO}}^{(0)}$, $n_{1,\text{RC}}^{(0)},n_{1,\text{IC}}^{(0)}$.
   \end{description}
 \end{proof}

\section{The number $k_{ij}$ of distinct points in distinct planes of $\PG(3,q)$. Structure of the point-plane incidence matrix}\label{sec_points in planes}
Recall that, by Lemma \ref{lem4_pointinplane}, we have the same number $r_{ij}$ of planes from an orbit $\N_i$ through every point of an orbit  $\M_j$, and vice versa,  the number $k_{ij}$ of points from $\M_j$ in a plane of $\N_i$ is the same for all planes of $\N_i$.
\begin{thm}\label{th6_pointinplane}
For $i,j=1,\ldots,5$, the following holds:
\begin{align}\label{eq6_ones}
 &   k_{ij}\cdot\# \N_i=r_{ij}\cdot\#\M_j\,;\displaybreak[3]\\
&\sum_{j=1}^5 r_{ij}=\sum_{i=1}^5 k_{ij}=q^2+q+1.\label{eq6_pointsinplane_sum}
\end{align}
\end{thm}
\begin{proof}
The cardinality of the multiset consisting of the points of $\M_j$ in all planes of $\N_i$ is equal to $r_{ij}\cdot\#\M_j$. By Lemma \ref{lem4_pointinplane}, every plane of $\N_i$ contains the same number of points of $\M_j$. Thus, $k_{ij}= \frac{r_{ij}\cdot \#\M_j}{\#\N_i}. $

Relation \eqref{eq6_pointsinplane_sum} holds as  $\PG(3,q)$ is partitioned  under $G_q$ into 5 orbits $\M_j$ and $\N_i$.
 \end{proof}

The values $r_{ij}$ and $k_{ij}$ are collected in Tables \ref{tab1} and~\ref{tab2}.

Recall that the point-plane incidence matrix  of the $\PG(3,q)$ consists of 25 submatrices~$\I_{ij}$.
The submatrix $\I_{ij}$ has size $\#\N_i\times\#\M_j$; it contains $k_{ij}$ ones in every row and $r_{ij}$ ones in every column, see \eqref{eq6_ones}.
\begin{proposition}
   For $q\not\equiv0\pmod3$,   $\I_{ij}^{tr}=\I_{ji}$ up to rearrangement of rows and columns. Also,
   \begin{align*}
    \#\N_i=\#\M_i,~\#\M_j=\#\N_j,~k_{ij}=r_{ji}, ~r_{ij}=k_{ji},~i,j\in\{1,\ldots,5\}.
   \end{align*}
\end{proposition}
\begin{proof}
    The assertion follows from Theorem \ref{th2_HirsCor4,5}(D), see \eqref{eq2_null_pol}.
 \end{proof}

 In the next proposition we use Notation \ref{notation_1} for a $t$-$(v,k,\lambda)$ design. The definitions of $m$-multiple and decomposable $2$-$(v,k,\lambda)$ designs can be found in \cite[Section II.1.1.6]{HandbCombDes}.
\begin{proposition}
\begin{description}
  \item[(i)] The submatrix $\I_{21}$ is an incidence matrix of a $2$-multiple decomposable $2\text{-}(q+1,2,2)$ design. It can be viewed an union of two incidence matrices of a  $2\text{-}(q+1,2,1)$ design.
  \item[(ii)] The submatrix $\I_{31}$ is an incidence matrix of a $3\text{-}(q+1,3,1)$ and a $2\text{-}(q+1,3,q-1)$ designs.
 \end{description}
\end{proposition}
\begin{proof}
\begin{description}
  \item[(i)]
  Rows and columns of the $2\binom{q+1}{2}\times(q+1)$ submatrix $\I_{21}$ are labeled, respectively, by $2_\C$-planes and $\C$-points, see Tables \ref{tab1} and \ref{tab2}. A column of $\I_{21}$ corresponds to a point of a design, i.e. $v=q+1$. A row corresponds to a $k$-block. By the definition of a $2_\C$-plane, every row contains exactly two ones, i.e. $k=2$. We consider a $t$-$(q+1,2,\lambda)$ design. Let $t=2$. Each two points of $\C$ generate a real chord. By Lemma~\ref{lem4_3 2}, there are exactly two $2_\C$-planes through a real chord. So, $\lambda=2$.

    We partition the $2\binom{q+1}{2}$-set of $2_\C$-planes into two $\binom{q+1}{2}$-subsets $B_1$ and $B_2$ as follows: for each real chord we place one of two $2_\C$-planes through it to $B_1$ and other one to $B_2$. By Theorem \ref{th2_HirsCor4,5}(C(ii)), no two real chords of $\C$ meet off $\C$. Therefore, $B_i$ gives 2-blocks of a $2\text{-}(q+1,2,1)$ design, $i=1,2$.
       \item[(ii)] We act similarly to part (i). Rows and columns of the $\frac{1}{6}(q^3-q)\times(q+1)$ submatrix $\I_{31}$ are labeled, respectively, by $3_\C$-planes and $\C$-points. By the definition of a $3_\C$-plane, every row contains exactly three ones. A column (resp. row) of $\I_{31}$ corresponds to a point (resp. a $3$-block) of the design.  Thus, $v=q+1$, $k=3$.

We consider a $t$-$(q+1,3,\lambda)$ design.
       For $t=3$, note that there is one and only one $3_\C$-plane through any three points of $\C$, i.e. $\lambda=1$.

Let $t=2$. There is a real chord through each two points of $\C$. By Lemma~\ref{lem4_3 2}, the number of $3_\C$-planes through a real chord is equal to $q-1$. So, $\lambda=q-1$. In addition, we note that, by \cite[Section II.4.2, Theorem 4.8]{HandbCombDesCh4}, every $3$-$(q+1,3,1)$ design is also a $2$-$(q+1,3,q-1)$ design.
\end{description}

\end{proof}

\begin{corollary}
From Tables \emph{\ref{tab1}} and~\emph{\ref{tab2}} the following holds:
\begin{description}
  \item[(i)]For $q\equiv0\pmod3$, up to rearrangement of rows and columns, we have
  \begin{align*}
    \I_{41}^{tr}=\I_{14},~ \I_{41}^{tr}=\I_{15},~ \I_{42}^{tr}=\I_{14},~ \I_{42}^{tr}=\I_{15}.
\end{align*}
  \item[(ii)] If $\#\N_i=\#\M_j$, then  the submatrix $\I_{ij}$ gives rise to a symmetric tactical configuration with $ k_{ij}=r_{ij}$. This holds  for $\I_{ii}$, $i=1,\ldots,5$, when $q\not\equiv0\pmod3$ and for $\I_{44},\I_{45}$ when $q\equiv0\pmod3$.
 \end{description}
\end{corollary}

\begin{proposition}
Let $q\equiv \xi\pmod3$. Let $i=1,\ldots,5$. Up to rearrangement of rows and columns, the following holds:
 \begin{description}
   \item[(i)]  The submatrix $\I_{i1}$ for $\xi=-1,1$ and for $\xi=0$ is the same;
   \item[(ii)] The submatrix $\I_{i4}$ for $\xi=1$ is the same as the submatrix $\I_{i5}$ for $\xi=0$;
   \item[(iii)]  The submatrix $\I_{i4}$ for $\xi=-1$ and for $\xi=0$ is the same.
 \end{description}
\end{proposition}

\begin{proof} The assertion (i) is clear. Regarding (ii) and (iii),
  by Theorem \ref{th2_HirsCor4,5}(B), we have $\M_{4}=\{\text{IC-points}\}$ for $\xi=1$ and $\M_{5}=\{\text{IC-points}\}$ for $\xi=0$. Also, $\M_{4}=\{\text{RC-points}\}$ for $\xi=-1$ as well as for $\xi=0$. Finally, see Theorems \ref{th5_nd1_1} and \ref{th5_RC}.
\end{proof}

\begin{thm}\label{th6_q_2_3_4}
 Let the orbits $\N_i$ and $\M_j$ be as in Theorem \emph{\ref{th2_HirsCor4,5}(B)}, see \eqref{eq2_plane orbit_gen}--\eqref{eq2_point_orbits_j=0}. For the twisted cubic $\C$ of \eqref{eq2_cubic} the following holds:
  \begin{description}
    \item[(i)] Let $q=2$. Under the action of the group $G_2\cong\mathbf{S}_3\mathbf{Z}_2^3$ fixing $\C$, there are four orbits $\widehat{\N}_i$ of planes and four orbits $\widehat{\M}_j$ of points where
    \begin{align}
      &\widehat{\N}_1=\N_1\cup\N_4,~ \widehat{\N}_2=\N_2,~ \widehat{\N}_3=\N_3, ~\widehat{\N}_4=\N_5; \label{eq6_q=2}\displaybreak[3]\\
      &\widehat{\M}_1=\M_1,~\widehat{\M}_2=\M_2\cup\M_5, ~\widehat{\M}_3=\M_3, \widehat{\M}_4=\M_4.\notag
    \end{align}
    The subgroup $\mathbf{S_3}\cong PGL(2,2)$ of $G_2$ partitions $\PG(3,2)$ into the orbits $\N_i$ and $\M_j$  as in Theorem \emph{\ref{th2_HirsCor4,5}(B)} for $q\not\equiv0\pmod3$. In this case, the point-plane incidence matrix has the form of Table \emph{\ref{tab1}}.
    \item[(ii)] Let $q=3$. Under the action of the group $G_3\cong\mathbf{S}_4\mathbf{Z}_2^3$ fixing $\C$, there are orbits $\widehat{\N}_i$ and $\widehat{\M}_j$ as in \eqref{eq6_q=2}.
     The subgroup $\mathbf{S_4}\cong PGL(2,3)$ of $G_3$ partitions $\PG(3,3)$ into the orbits $\N_i$ and $\M_j$  as in Theorem \emph{\ref{th2_HirsCor4,5}(B)} for $q\equiv0\pmod3$; the point-plane incidence matrix has the form of Table \emph{\ref{tab2}}.
    \item[(iii)]Let $q=4$. Under the action of the group $G_4\cong\mathbf{S}_5\cong P\Gamma L(2,4)$ fixing $\C$,  there are orbits $\N_i$ and $\M_j$  as in Theorem \emph{\ref{th2_HirsCor4,5}(B)} for $q\not\equiv0\pmod3$. In this case, the point-plane incidence matrix has the form of Table \emph{\ref{tab1}}.
  \end{description}
\end{thm}

\begin{proof}
  The groups $G_i$ are given in Theorem \ref{th2_HirsCor4,5}(A). The rest of the assertions are obtained by computer search using
the MAGMA computational algebra system \cite{MAGMA}.
\end{proof}

\section{The twisted cubic as a multiple covering code and a multiple 2-saturating set}\label{sec_multipl}
For $\rho=2$ and $N=3$, Definition \ref{def2_M1-M3} can be viewed as follows.
\begin{definition}
   Let $S$ be a subset of points of $PG(3,q)$. Then $S$
is said to be $(2,\mu )$-saturating if:
\begin{description}
\item[(M1)] $S$ generates $\PG(3,q)$;
\item[(M2)] there exists a point $Q$ in $\PG(3,q)$ which does not belong to
any bisecant line of~$S$;
\item[(M3)] every point $Q$ in $\PG(3,q)$ not belonging to any bisecant line of
$S$ is such that the number
of planes through  three points of $S$
containing $Q$ is at least $\mu $.
\end{description}
\end{definition}
\begin{thm}\label{th7_twcub=2musat}
The twisted cubic $\C$ of \eqref{eq2_cubic} is a minimal $(2,\mu )$-saturating $(q+1)$-set with $\mu$ as in~\eqref{eq3_mu}.
\end{thm}

\begin{proof}
\begin{description}
\item[(M1)]
 Any 4 points of $\C$ generate $\PG(3,q)$.
\item[(M2)] Apart from RC-points, all points  off $\C$ do not belong to
any bisecant line of $\C$.
\item[(M3)] Recall that  $n_{3,\bullet}^{(\xi)}$ is the number of $3_\C$-planes through a point of the type~$\bullet$. By Theorem \ref{th3_main} and Tables \ref{tab1} and \ref{tab2}, among points not lying on real chords the smallest value of $n_{3,\bullet}^{(\xi)}$ is $n_{3,\text{T}}^{(\ne0)}=(q^2-3q+2)/6$ if $q\not\equiv 0\pmod 3$ or
$n_{3,\text{TO}}^{(0)}=(q^2-3q)/6$ if $q\equiv 0\pmod3$.
\end{description}

It can be easily seen that  $\C$ is a \emph{minimal} $(2,\mu )$-saturating set.
 \end{proof}

\begin{thm}
    Let $\mu$ be as in \eqref{eq3_mu}. Let $C$ be the  code  associated with the twisted cubic $\C$ of \eqref{eq2_cubic}. Then
    \begin{description}
      \item[(i)]
    The code $C$ is a $[q+1,q-3,5]_q3$ quasi-perfect GDRS code of covering radius $R=3$ and, moreover, $C$ is a $(3,\mu)$-MCF code.
      \item[(ii)] The $\mu$-density $\gamma _{\mu }(C,3,q)$ of the code $C$ tends to $1$ from above when $q$ tends to infinity, i.e.
          \begin{align}\label{eq7_limit}
            \lim_{q\rightarrow\infty}\gamma _{\mu }(C,3,q)=1,~\gamma _{\mu }(C,3,q)>1.
          \end{align}
          Thereby, we have an asymptotical optimal collection of MCF codes.
    \end{description}
\end{thm}

\begin{proof}
    \begin{description}
      \item[(i)]
The twisted cubic is a normal rational curve. It is well known that a normal rational curve in $\PG(N,q)$ gives rise to a $[q+1,q-N,N+2]_q$ GDRS code. Also, by Proposition~\ref{prop2_connection} and Theorem~\ref{th7_twcub=2musat}, $C$ is a $(3,\mu)$-MCF code.
      \item[(ii)] Since $d(C)=
2R-1$, we have, by \eqref{eq2_mudens_d>=2R-1},
      \begin{align*}
&\gamma _{\mu }(C,3,q)=\frac{{\binom{q+1}{{3}}} (q-1)^{R}-{\binom{5}{{2}
}} (q-1)\binom{q+1}{5}}{\mu \left(q^{4}-1-(q^2-1)-\binom{q+1}{2}(q-1)^{2}\right)},
\end{align*}
where  $A_{2R-1}(C)=A_d(C)=(q-1)\binom{n}{d}$ as $C$ is an MDS code \cite{MWS,Roth}. After simple transformations, for $\mu=\frac{q^2-3q+2}{6}$, we obtain
\begin{align*}
\gamma_{\mu }(C,3,q)=\frac{\frac{1}{12}q^6 - \frac{1}{2}q^4 + \frac{1}{3}q^3 + \frac{5}{12}q^2 - \frac{1}{3}q}{\frac{1}{12}q^6- \frac{3}{4}q^4 + \frac{2}{3}q^3 + \frac{2}{3}q^2 - \frac{2}{3}q}
\end{align*}
          whence \eqref{eq7_limit} immediately follows. For $\mu=\frac{q^2-3q}{6}$ the proof is the same.
    \end{description}
 \end{proof}

\begin{rmk}
The Newton radius of a code is the largest Hamming weight of a uniquely correctable error, see \cite{GabKloeNewt,HelKlNewt} and the references therein. The $[q+1,q-3,5]_q3$ code $C$ associated with the cubic $\C$ of \eqref{eq2_cubic} corrects all double errors. In the geometrical language, a double error represents an RC-point as a linear combination of two $\C$-points. A point $A$ off $\C$ that does not lie on a real chord can be represented by a linear combination of three $\C$-points. This corresponds to a triple error. On the other hand, these three points generate a $3_\C$-plane in which $A$ lies. By Tables \ref{tab1} and \ref{tab2}, every point off $\C$ lies on more than one $3_\C$-plane. This means that distinct triple errors can give the same result, therefore triple errors are not uniquely correctable. Thus, Newton radius of $C$ is equal to two.
\end{rmk}

\end{document}